\documentclass[11pt]{amsart}
\usepackage[utf8]{inputenc}
\usepackage{bm}
\usepackage{fontenc}
\usepackage{amsfonts}
\usepackage{amssymb}
\usepackage{amsmath}
\usepackage{amsthm}
\usepackage{mathtools}
\usepackage{enumerate}
\usepackage{hyperref}
\usepackage{enumitem}
\usepackage{mathrsfs}
\usepackage{tikz}
\usetikzlibrary{calc}
\usepackage{marginnote}
\usepackage{xcolor,enumitem}
\usepackage{soul}
\usepackage{tikz}
\usepackage[all,cmtip]{xy}

\newcommand{\e}{\varepsilon}

\newcommand{\NN}{\mathbb{N}}
\newcommand{\zet}{\mathbb{Z}}

\newtheorem*{rigprob*}{Rigidity Problem for uniform Roe Algebras}
\newtheorem*{rigprobcorona*}{Rigidity Problem for uniform Roe Coronas}

\newcommand{\cstar}{$\mathrm{C}^*$}

\newcommand{\diag}{\mathrm{diag}}

\newcommand{\Aut}{\mathrm{Aut}}

\newtheorem{theorem}{Theorem}[section]
\newtheorem*{theorem*}{Theorem}
\newtheorem{proposition}[theorem]{Proposition}
\newtheorem{problem}[theorem]{Problem}
\newtheorem*{proposition*}{Proposition}
\newtheorem{lemma}[theorem]{Lemma}
\newtheorem*{lemma*}{Lemma}

\newtheorem*{corollary*}{Corollar}

\newtheorem*{fact*}{Fact}
\theoremstyle{definition}
\newtheorem{definition}[theorem]{Definition}
\newtheorem*{definition*}{Definition}

\newtheorem*{claim*}{Claim}

\newtheorem*{conjecture*}{Conjecture}

\theoremstyle{remark}

\newtheorem*{example*}{Example}
\newtheorem{remark}[theorem]{Remark}
\newtheorem*{remark*}{Remark}

\newtheorem*{note*}{Note}
\newtheorem*{question*}{Question}

\newtheorem*{acknowledgements*}{Acknowledgements}

\newcommand{\norm}[1]{\left\lVert #1 \right\rVert}

\DeclareMathOperator{\supp}{supp}

\newcommand{\Fraisse}{Fra\"{i}ss\'e }

\newcounter{my_enumerate_counter}
\newcommand{\pushcounter}{\setcounter{my_enumerate_counter}{\value{enumi}}}
\newcommand{\popcounter}{\setcounter{enumi}{\value{my_enumerate_counter}}}

\usepackage{enumitem}

\makeindex
\begin{document}

 \title{Fra\"{i}ss\'e theory in operator algebras}

\author[A. Vignati]{Alessandro Vignati}
\address[A. Vignati]{
Institut de Math\'ematiques de Jussieu - Paris Rive Gauche (IMJ-PRG)\\
Universit\'e Paris Cit\'e\\
B\^atiment Sophie Germain\\
8 Place Aur\'elie Nemours \\ 75013 Paris, France}
\email{ale.vignati@gmail.com}
\urladdr{http://www.automorph.net/avignati}

\date{\today}%
 
\maketitle
\begin{abstract}
We overview the development of \Fraisse theory in the setting of continuous model theory, and some of the its recent applications to \cstar-algebra theory and functional analysis.
\end{abstract}

\section{Introduction}
This short note aims to review the recent progress in the study of the applications of \Fraisse theory to continuous model theory, and in particular to operator algebras.

\Fraisse theory, after \cite{Fraisse}, is an area of mathematics at the crossroads of combinatorics and model theory. In the discrete setting, \Fraisse theory studies countable homogeneous structures and their relations with properties of their finitely-generated substructures. Given a countable structure $\mathcal M$, its \emph{age} is the class of all finitely-generated substructures of $\mathcal M$. Ages of countable homogeneous structures satisfy certain combinatorial properties (notably amalgamation properties) making them \Fraisse classes. Conversely, given any \Fraisse class, one constructs a countable homogeneous structure, its \Fraisse limit, a unique (up to isomorphism) structure with the given class as its age. Many interesting objects across mathematics (in group theory, graph theory, and topology) were identified as \Fraisse limits (see for example \cite[Chapter 7]{Hodges} and \cite{IrwSol.Pseudo}).

\Fraisse theory has found applications in several areas of mathematics. An example deserving to be mentioned is topological dynamics: In the seminal \cite{KPT}, Kechris, Pestov, and Todor\v{c}evi\`c established what is now known as the KPT-correspondence, linking Ramsey theoretical property of a \Fraisse class and dynamical properties of the automorphism group of its \Fraisse limit. The KPT correspondence is used to compute, and study properties of, universal minimal flows of automorphism groups of \Fraisse limits via Ramsey theoretic conditions on the \Fraisse class of interest. (The universal minimal flow of a topological group $G$ is a compact dynamical system of $G$ canonically associated with the group which is of great interest in topological dynamics, see~\S\ref{s.ramsey}.) The KPT correspondence and its extensions were used to study automorphism groups of \Fraisse limits arising from certain classes of structures such as graphs, directed graphs, posets, lattices, and so on. 

\Fraisse theoretic ideas were discussed in analysing the class of finite-dimensional Banach spaces and its amalgamation properties by Kubi\`s and Solecki in \cite{KubisSolecki}. Before a proper formalisation of \Fraisse theory for continuous structures, objects of \Fraisse theoretic nature were studied via topological and categorical methods, (see e.g. \cite{Kubis.Metric} and \cite{KwiaBart}). A systematic approach to \Fraisse theory in the setting of continuous logic was discussed in \cite{Schoretsanitis} and then introduced in \cite{BY:Fraisse}. Notably, many structures in functional analysis have been recognized to be of \Fraisse theoretic nature, for example in the work of Ben Yaacov and collaborators (\cite{BenYaacov.Linear,BenYHenson}), Kubi\`s and his collaborators (\cite{KubisGarbu,KubisKwia,Kubis.Metric,KubisFr2}), Conley and T\o rnquist (\cite{TornConley.Fraisse}), and Lupini (\cite{Lupini.fa,Lupini.NCGurarij}). Operator algebra examples were treated first in \cite{EFH:Fraisse}, then in Masumoto's work (\cite{Masumoto.FraisseZ} and \cite{Masumoto.Real}), and further in \cite{GhasemiKubis} and \cite{JacelonVignati}.

Here, we follow the approach of \cite{BY:Fraisse}. There, Ben Yaacov treats formally classes of metric structures, and gives the definition of \Fraisse classes and \Fraisse limits. Fix a language for metric structures $\mathcal L$. Informally, the \Fraisse theorem for continuous logic asserts that if $\mathcal K$ is a class of finitely-generated $\mathcal L$-structures which is hereditary, directed, satisfies reasonable separability conditions, and has the \emph{near amalgamation property}, then there is a unique separable $\mathcal L$-structure $M$ which has $\mathcal K$ as its age and is ultrahomogeneous, meaning that, given a finite tple $\bar a$ in $M$ the closure of the automorphism orbit of $\bar a$ contains all tuples with the same quantifier-free type of $\bar a$. A version of the \Fraisse theorem for nonhereditary classes, restricting the amount of homogeneity one obtains, also holds.

Pivotal objects in operator algebras and functional analysis have been recognised as \Fraisse limits of appropriate classes of structures. We give two examples of crucial importance in their respective areas: the Jiang-Su algebra $\mathcal Z$ and the Poulsen simplex $\mathbf P$.

The Jiang-Su algebra $\mathcal Z$ was constructed by Jiang and Su in \cite{JiangSu} as a limit of subhomogeneous \cstar-algebras known as dimension drop algebras (see \S\ref{ss.blocks}). $\mathcal Z$ has been at the center of the classification program for \cstar-algebras. In a precise sense, $\mathcal Z$ is the smallest possible \cstar-algebra that tensorially absorbs itself in a very strong way, and so, from this perspective, is the most natural \cstar-analogue of the hyperfinite II$_1$ factor $\mathcal R$, at the center of Connes' Fields Medal work on von Neumann algebras (\cite{Connes.Class}).
A \cstar-algebra $A$ which satisfies $A\otimes \mathcal Z\cong A$ is called $\mathcal Z$-absorbing. By a celebrated result of many hands, $\mathcal Z$-absorption is a regularity condition, capable of detecting those (suitable) \cstar-algebras which can be classified. More precisely, separable $\mathcal Z$-absorbing \cstar-algebras which are in addition unital, simple, and nuclear (and satisfy the technical condition known as UCT) are precisely those which can be classified by their Elliott invariant (see \cite{GongLinNiu,GongLinNiu.Class1} or \cite[Corollary D]{CETWW}). Further, the Toms-Winter conjecture asserts that $\mathcal Z$-absorption is linked to other well-behavedness and regularity conditions related to low dimensionality (in terms of the nuclear dimension of Winter and Zacharias, \cite{WinterZac.Dimnuc}) and perforation properties of the Cuntz semigroup (a refined version of $K$-theory defined by Cuntz). $\mathcal Z$ has been the focus of prominent research in the last 20 years for its unique and peculiar properties (see e.g. \cite{Winter.ICM}), and has been approached by \Fraisse theoretic methods in \cite{EFH:Fraisse}, and then in \cite{Masumoto.FraisseZ} and \cite{Masumoto.Real}. $\mathcal Z$'s nonunital cousins, $\mathcal W$ and $\mathcal Z_0$, play the role of $\mathcal Z$ in the classification program of nonunital \cstar-algebras; they were approached by \Fraisse theoretic methods in \cite{JacelonVignati}.

The theory of Choquet simplices lies at the intersection of Banach space theory, convex geometry, and the study of Polish spaces. In fact, Kadison's representation theorem (\cite[Theorem II.1.8]{Alfsen.Book}) gives a contravariant equivalence between the categories of compact convex sets (with continuous affine maps) and that of function systems (with positive linear maps). 
 Choquet simplices correspond to function systems that are moreover Lindenstrauss spaces (preduals of $L_1$ spaces), and are therefore connected to measure theory, and consequently the whole of functional analysis. 
The Poulsen simplex $\mathbf P$ was introduced in \cite{Poulsen} as a metrizable Choquet simplex whose extreme points are dense. Despite the fact that Poulsen's original construction has many degree of freedom, in \cite{Lind-Olsen-Sternfeld} (see also \cite{Olsen}), $\mathbf P$ was shown to be the unique metrizable Choquet simplex whose extreme points are dense (up to affine homeomorphism). The study of $\mathbf P$ has been inspired by that of Gurarij's universal and homogeneous Banach space $\mathbf G$ (see \S\ref{s.fa}). In fact, as specifically stated in \cite{Lind-Olsen-Sternfeld}, the idea to prove uniqueness of the Poulsen simplex came directly from Lusky's proof of the uniqueness of Gurarij's generic Banach space (\cite{Lusky}). $\mathbf P$ has many homogeneity and universality properties (which can be reconstructed by \Fraisse theoretic methods, as noticed in \cite{Lupini.fa}).

Our overview of the applications of \Fraisse theory to the study of operator algebras and functional analysis is structured as follows. In \S\ref{s.prel}, we introduce \Fraisse classes and their limits. In \S\ref{s.opalg} we describe the findings of \cite{EFH:Fraisse}, \cite{Masumoto.FraisseZ} and \cite{JacelonVignati}, where the \cstar-algebras $\mathcal Z$, $\mathcal W$, $\mathcal Z_0$ and the unital separable UHF algebras were recognised as \Fraisse limits. We also see how to construct surjectively universal AF algebras as \Fraisse limits (\cite{GhasemiKubis}). In \S\ref{s.fa} we overview the applications of \Fraisse theory to functional analysis: we introduce the Gurarij spaces $\mathbf G$ and $\mathbf {NG}$, as well as the Poulsen simplex $\mathbf P$ and its noncommutative version, the noncommutative Choquet simplex $\mathbf {NP}$, and we describe how there can be viewed as \Fraisse limits of suitable classes. Lastly, in \S\ref{s.ramsey}, we go through some of the connections between Ramsey theory and topological dynamics of automorphisms groups of \Fraisse structures, describing some of the findings of \cite{BartosovaRamsey1} and \cite{BartosovaRamsey2}. 

\vspace{3pt}

In case the reader lacks the preliminary notions, it is redirected to \cite{Blackadar.OA} or Szab\`o's article for an introduction to \cstar-algebras, to \cite{BYBHU} or Hart's article for an introduction to continuous model theory, and to \cite{bourbaki} or Sinclair's article for the model theory of \cstar-algebras.

\subsection*{Acknowledgements}
The author is partially supported by the ANR project AGRUME (ANR-17-CE40-
0026) and by an Emergence en Recherche grant from the Universit\'e Paris Cit\'e - IdeX.

\section{\Fraisse classes and their limits}\label{s.prel}
In this section, we introduce \Fraisse classes and their limits. We follow an approach similar to that of \cite{BY:Fraisse}.

\begin{definition} 
A \emph{language} for metric structures $\mathcal L$ is a set of relation symbols $(R_i)_{i\in I}$ and function symbols $(f_j)_{j\in J}$. To each symbol of the language are attached a natural number (its arity) and a modulus of uniform continuity. $0$-ary functions are constants. An $\mathcal L$-\emph{structure} is a complete metric space $(X,d)$ together with interpretations for relation and function symbols.
\end{definition}

We always assume that our languages contain a special binary symbol, which is
interpreted by the distance function in all structures, analogously as one interprets equality in the discrete setting.

If $A$ is an $\mathcal L$-structure, $n\in\NN$, and $\bar a\in A^n$, we denote by $\langle \bar a\rangle$ the smallest $\mathcal L$-substructure containing $\bar a$ (i.e., it is stable under interpretations of all functions in $\mathcal L$). 
\begin{definition}
An $\mathcal L$-\emph{embedding} is a map which commutes with interpretations of the symbols in $\mathcal L$.
\end{definition}

We fix, for the remaining part of this section, a separable language for metric structures $\mathcal L$. All structures (substructures, embeddings, ...) are implicitly $\mathcal L$-structures ($\mathcal L$-substructures, $\mathcal L$-embeddings, ...).

\begin{definition}
Let $\mathcal K$ be a class of finitely-generated $\mathcal L$-structures. $\mathcal K$ is said to have:
\begin{itemize}
\item the \emph{Hereditary Property} (HP) if whenever $A\in\mathcal K$ and $B$ is a substructure of $A$, then $B\in\mathcal K$;
\item the 
 \emph{Joint Embedding Property} (JEP) if for all $A,B\in\mathcal K$ there is $C\in\mathcal K$ and two embeddings $\phi\colon A\to C$ and $\psi\colon B\to C$;
\item the 
 \emph{Near Amalgamation Property} (NAP) if whenever we are given $A$, $B$, and $C$ in $\mathcal K$, a finite set $F\subseteq A$ and $\e>0$, then for every pair of embeddings $\phi_1\colon A\to C$ and $\phi_2\colon B\to C$ there is $D\in\mathcal K$ and two embeddings $\psi_1,\psi_2\colon C\to D$ such that 
\[
d(\psi_1\circ\phi_1(a),\psi_2\circ\phi_2(a))<\e \text{ for all } a\in F,
\]
where $d$ is the distance predicate as computed in $D$.
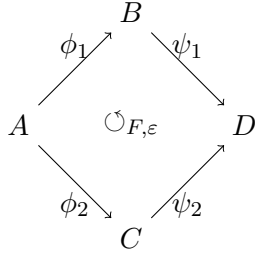
\begin{center}
\begin{tikzpicture}[scale=0.5]
\node (A) at (0,-3) {$A$};
\node (C) at (3,0) {$B$};
\node (D) at (3,-6) {$C$};
\node (E) at (6,-3) {$D$};
\node (F) at (3,-3) {$\circlearrowleft$${}_{F,\e}$};
\draw (A) edge[->] node [above] {$\phi_1$} (C);
\draw (A) edge[->] node [below] {$\phi_2$} (D);
\draw (C) edge[->] node [above] {$\psi_1$} (E);
\draw (D) edge[->] node [below] {$\psi_2$} (E);
\end{tikzpicture}
\end{center}

\end{itemize}
\end{definition}
The final condition one must impose is the analogue of countability in the discrete setting. 
If $\bar a=(a_1,\ldots,a_n)$ and $\bar b=(b_1,\ldots,b_n)$ are tuples in a common structure $A$ with distance $d$, we abuse of notation and let
\[
 d(\bar a,\bar b):=\max_{i}d(a_i,b_i).
 \]
 
\begin{definition}
Let $\mathcal K$ be a class of finitely-generated $\mathcal L$-structures satisfying JEP and NAP. We denote by $\mathcal K_n$ the set of all pairs $(A,\bar a)$ where $\bar a\in A^n$ generates $A$, and $A\in\mathcal K$.

Define a pseudometric $d_n$ on $\mathcal K_n$ by setting
\[
d_n((A,\bar a),(B,\bar b))=\inf_{\phi,\psi}d(\phi(\bar a),\psi(\bar b)),
\]
where $\phi$ and $\psi$ range over all possible embeddings of $\langle a\rangle $ and $\langle b\rangle$ in a common $C\in\mathcal K$.

$\mathcal K$ is said to have the \emph{Weak Polish Property} (WPP) if each $(\mathcal K_n,d_n)$ is a separable space.
\end{definition}

\begin{definition}
A class of finitely-generated structures satisfying JEP, NAP, and the WPP is said to be a 
 \emph{\Fraisse class}. If the class satisfies in addition HP, it is said to be a \emph{hereditary \Fraisse class}.
\end{definition}

\begin{definition}
Let $\mathcal K$ be a class of finitely-generated $\mathcal L$-structures. A separable $\mathcal L$-structure is said to be a $\mathcal K$-\emph{structure} if it is the inductive limit of structures in $\mathcal K$. 

A $\mathcal K$-structure $M$ is a 
 \emph{\Fraisse limit} of $\mathcal K$ if 
\begin{enumerate}
\item $M$ is $\mathcal K$-\emph{universal}: each $\mathcal K$-structure admits an embedding into $M$;
\item $M$ is \emph{approximately} $\mathcal K$-\emph{homogeneous}: for every $A=\langle \bar a\rangle\in\mathcal K$, $\e>0$, and embeddings $\phi_1,\phi_2\colon A\to M$ there is $\sigma$, an automorphism of $M$, such that 
\[
d(\sigma(\phi_1(\bar a)),\phi_2(\bar a))<\e,
\] 
where $d$ is the distance in $M$.
\end{enumerate}
\end{definition}

The following result is the main result relating \Fraisse classes and homogeneous structures. See \cite{BY:Fraisse} for a proof.
\begin{theorem}
Every \Fraisse class has a \Fraisse limit which is unique up to isomorphism.
\end{theorem}
\begin{definition}
Let $\mathcal K$ be a \Fraisse class with \Fraisse limit $M$. Let $(A_n,\varphi_n)$ be a sequence of $\mathcal K$ objects together with embeddings $\varphi_n\colon A_n\to A_{n+1}$. The sequence $(A_n,\varphi_n)$ is said to be 
 \emph{generic} if $\lim (A_n,\varphi_n)\cong M$.
\end{definition}

We conclude this section with a few remarks.

\begin{remark}\label{remark1}
\begin{itemize}
\item Our approach is slightly different from that of \cite{BY:Fraisse}, which is more technical and involved, as translating notions and theorems from discrete model theory to the continuous setting has to be done always extremely carefully. The setting displayed suffices for all the applications we are going to review. (For examples, notice that the classes introduced here are incomplete in the sense of \cite{BY:Fraisse}. The completions of our classes will include their \Fraisse limits.)
\item
A second and simpler proof of the \Fraisse theorem is due to Todor Tsankov. The author (and others) would like very much to see such proof published. 
\item The \Fraisse classes appearing in \S\ref{s.opalg}, apart from the ones defining Bing's pseudoarc $\mathbb P$ and the hyperfinite II$_1$ factor $\mathcal R$, are not hereditary. This is due to the fact that the class of finitely-generated subalgebras of a given \cstar-algebra $A$ is often complicated, and very large. In fact, for the age of a \cstar-algebra $A$, amalgamation indeed fails most of the time. All examples considered in \S\ref{s.fa} are hereditary classes by their very nature.
\item While in the discrete setting many (though not all) well-known \Fraisse limits have theories with quantifier elimination, this is not true in the operator algebraic setting. It was proved in \cite{eagle2015quantifier} that the only separable infinite-dimensional unital \cstar-algebra having quantifier elimination in the usual language for \cstar-algebras is $C(2^\NN)$, the algebra of continuous functions on the Cantor set. On the other hand, objects from functional analysis arising as \Fraisse limits (such as the Urysohn space, or the Hilbert space $\ell_2(\NN)$), do have quantifier elimination in their natural languages.
\end{itemize}
\end{remark} 

\section{\Fraisse theory in operator algebras}\label{s.opalg}

Here, we describe the main applications of \Fraisse theory to operator algebras. Most of the classes we will be concerned with are not hereditary classes, as mentioned in Remark~\ref{remark1}. This is due to the fact that infinite-dimensional \cstar-algebras are often singly generated (see e.g. \cite{ThielWinter}). (In the unital abelian setting finitely-generated \cstar-algebras correspond to those whose spectrum is of finite covering dimension.) With plenty of finitely-generated structures to take care of, obtaining approximate amalgamation (even allowing natural expansions of the language) is therefore quite difficult. For this reason, we only consider classes that are made of suitably `small' algebras.

We will (almost) always work in the language of \cstar-algebras as introduced in \cite{FHS.II} (see also \cite{bourbaki}), or some of its natural expansions.

\subsection{Abelian examples}
We start by recording easy examples of topological nature.

\subsubsection*{The Cantor set}
The following is equivalent to the fact that the class of finite sets together with surjective maps forms a projective \Fraisse class\footnote{Projective \Fraisse theory can be seen as the dual version of the classical ones, where maps are often surjections. Projective \Fraisse theory have been used in topology (e.g., \cite{KwiaBart,IrwSol.Pseudo}) to recognise certain spaces are generic.} whose limit is the 
 Cantor set $2^\NN$.

\begin{theorem}
Let $\mathcal K$ be the class whose objects are finite-dimensional abelian \cstar-algebras and whose maps are unital injective $^*$-homomorphisms. Then $\mathcal K$ is a \Fraisse class whose \Fraisse limit is $C(2^\NN)$.
\end{theorem}
\subsubsection*{The pseudoarc}
We dualize the content of \cite[\S4.3]{Kubis.Metric}, where it was proved that Bing's 
 pseudoarc $\mathbb P$ is the projective \Fraisse limits of intervals (even though in the language of metric enriched categories). In particular, we show that the algebra $C(\mathbb P)$ is a \Fraisse limit in the setting of \S\ref{s.prel}.

First, a lemma. (See \cite{Homma} for the original version, and \cite{SikoZara} for a proof of the result as stated below.)
\begin{lemma}[Mountain Climbing Lemma]
Let $f$ and $g$ be continuous piecewise linear surjections $[0,1]\to[0,1]$ such that $f(0)=0=g(0)$ and $f(1)=1=g(1)$. Then there are surjections $f',g'\colon [0,1]\to[0,1]$ such that $f\circ f'=g\circ g'$.
\end{lemma} 
Dualizing the content of the Mountain Climbing Lemma, we get the following result. (Since the proof in the current language has not appeared anywhere, we sketch it.)
\begin{theorem}
Let $\mathcal K_{\mathbb P}$ be the class whose only object is $C([0,1])$ and whose maps are unital injective $^*$-homomorphisms. Then $\mathcal K_{\mathbb P}$ is a \Fraisse class.
\end{theorem}
\begin{proof}[Sketch of the proof]
The only nontrivial property to prove is the Near Amalgamation Property. Fix a finite $F\subseteq C([0,1])$, $\e>0$, and two injective $^*$-homomorphisms $\phi_1,\phi_2\colon C([0,1])\to C([0,1])$. Since injections between abelian \cstar-algebras correspond to surjections between the corresponding spaces, consider the continuous surjections $\xi_1,\xi_2\colon [0,1]\to [0,1]$ such that $\phi_i(f)(t)=f(\xi_i(t))$ for $i=1,2$ and $t\in [0,1]$. First, we find two continuous surjections $\xi_1'$ and $\xi_2'$ with the property that $\xi_i\circ\xi_i'(0)=0$ and $\xi_i\circ\xi_i'(1)=1$ for $i=1,2$. Since those $^*$-homomorphisms whose duals are piecewise linear are dense in point norm topology, and $F$ is finite, we can assume that each $\xi_i\circ\xi_i'$, for $i=1,2$, is piecewise linear. The duals of the amalgamating surjections provided by the Mountain Climbing Lemma give the NAP.
\end{proof}

To study the \Fraisse limit of the class $\mathcal K_{\mathbb P}$ we use Gelfand duality, and study the topological properties of its spectrum. 

A connected compact Hausdorff space is a \emph{continuum}. A continuum $X$ is \emph{chainable} if every open cover of $X$ can be refined by a finite open cover $U_1,\ldots,U_n$ with the property that $U_i\cap U_j\neq\emptyset$ if and only if $|i-j|\leq 1$. $X$ is \emph{indecomposable} if it cannot be written as the union of two of its proper subcontinua, and \emph{hereditary indecomposable} if each of its subcontinua is indecomposable. The only metrizable chainable hereditary indecomposable continuum is Bing's pseudoarc $\mathbb P$ (\cite{Bing1}).

\begin{theorem}
The \Fraisse limit of $\mathcal K_{\mathbb P}$ is Bing's pseudoarc $\mathbb P$.
\end{theorem}
\begin{proof}[Sketch of the proof] Let $X$ be the compact metrizable space such that the \Fraisse limit of $\mathcal K_{\mathbb P}$ is isomorphic to $C(X)$. That $X$ is chainable is obvious, being an inverse limit of chainable continua. Kubi\`s's argument (\cite[Lemma 4.21 and 4.22]{Kubis.Metric}), when dualised, gives that $X$ is hereditary indecomposable (in fact, he shows that $X$ is indecomposable and every proper subcontinuum of $X$ is homeomorphic to $X$ itself). The result follows from Bing's theorem \cite{Bing1}.
\end{proof}
\subsubsection*{Other examples}
Other topological spaces can be viewed as projective \Fraisse limits of suitable classes of spaces, once certain restrictions are given on the maps in consideration. Dually, their algebras of functions could be approached by a \Fraisse theoretic point of view as limit of suitable classes of abelian \cstar-algebras. 

Examples are the Menger curve $\mathbb M$ and the Lelek fan $\mathbb L$, treated as projective \Fraisse limits of finite graphs (with connected maps) and fans, both considered as one-dimensional compact spaces, see e.g. \cite{SolPana.Menger} and \cite{KwiaBart}.
Neither $C(\mathbb M)$ nor $C(\mathbb L)$ have been formally recognised as \Fraisse limits in the context of continuous logic, and more specifically operator algebras.
\begin{problem}
Is there a class of \cstar-algebras $\mathcal K_{\mathbb M}$ where embeddings are injective $^*$-homomorphisms, having $C(\mathbb M)$ as its \Fraisse limit? What about $C(\mathbb L)$?
\end{problem}

\subsection{Nonabelian \cstar-algebras}
In this section we summarise the work of \cite{EFH:Fraisse}, \cite{Masumoto.FraisseZ}, and \cite{JacelonVignati}, which recognised certain pivotal objects in the classification of nuclear \cstar-algebras as \Fraisse limits of suitable classes. If $n\in\mathbb N$, $M_n$ denotes the \cstar-algebra of $n\times n$ complex valued matrices, $M_n(\mathbb C)$.

\subsubsection{UHF algebras}\label{ss.UHF}
A separable \cstar-algebra $A$ is 
 \emph{Uniformly HyperFinite} (UHF) if it is a direct limit of matrix algebras (see \S7 of Szab\`o's article in this volume). Thanks to Glimm's theorem, unital UHF algebras are classified by their supernatural numbers, that is, by formal products $\bar p=\prod_{p\text{ prime} }p^{\ell_p}$, where $\ell_p\in\mathbb N\cup\{\infty\}$. If $A$ is a unital UHF algebra with $A=\lim_i M_{n_i}$, for each prime $p$, let 
\[
\ell_p=\sup \{r\mid p^r\text{ divides }n_i \text{ for some } i\}.
\]
This associates a supernatural number $\bar p_A$ to $A$, and one checks that 
\[
A\cong M_{\bar p_A}:=\bigotimes_{p \text{ prime }}\bigotimes_{j\leq \ell_p}M_p.
\]
\begin{definition}
Let $\bar p$ be a supernatural number. $\mathcal K_{\bar p}$ is the class of matrix algebras whose size divides $\bar p$, where maps are unital injective $^*$-homomorphisms.
\end{definition}
The fact that any two unital $^*$-homomorphisms between matrix algebras are unitarily equivalent immediately gives (exact!) amalgamation for $\mathcal K_{\bar p}$. As the limit of a sequence of objects in $\mathcal K_{\bar p}$ is automatically UHF, we have the following.
\begin{theorem}
For each supernatural $\bar p$, the class $\mathcal K_{\bar p}$ is a \Fraisse class whose \Fraisse limit is $M_{\bar p}$.
\end{theorem}

\subsubsection{The hyperfinite II$_1$ factor $\mathcal R$}
Analyzing UHF algebras, we treated matrix algebras as \cstar-algebras; here we consider them as tracial von Neumann algebras in the appropriate language (see \cite{FHS.II}). In this setting, the distance symbol is interpreted as the distance associated with the trace norm $\norm{x}_\tau=\tau(xx^*)^{1/2}$. We recall that a tracial von Neumann algebra $(M,\tau)$ is a factor if it has trivial center. Among separable (for $\norm{\cdot}_\tau$) factors, there is a unique amenable object, known as the II$_1$ factor $\mathcal R$. In this setting, amenability coincides with hyperfiniteness, where a II$_1$ factor is 
 hyperfinite if it can be approximated, in trace-norm, by matrix algebras (see e.g. \cite{Connes.Class}). The same amalgamation one proves for matrix algebras gives the following, which was stated as \cite[Theorem 3.5]{EFH:Fraisse}.
\begin{theorem}
In the language of tracial von Neumann algebras, the hyperfinite II$_1$ factor $\mathcal R$ is the \Fraisse limit of the \Fraisse class of matrix algebras.
\end{theorem}

\subsection{Finite-dimensional \cstar-algebras}
We have seen (\S\ref{ss.UHF}) that full matrix algebras have amalgamation, in the category of unital \cstar-algebras. One might wonder whether the same holds for finite-dimensional \cstar-algebras. If $F$ is a \cstar-algebra which is finite-dimensional as a vector space, then there are positive naturals $n_1,\ldots,n_k$ such that $F\cong\bigoplus_{i\leq k}M_{n_i}$. Limits of finite-dimensional \cstar-algebras are called approximately finite-dimensional (AF) algebras; these algebras were of great interest in the 70's and 80's, representing some of the first nontrivial examples of infinite-dimensional simple \cstar-algebras. They were classified by Elliott in \cite{Elliott.intertwining}.

A first guess in trying to construct a universal and somewhat homogeneous AF algebra would be to consider the class of all finite-dimensional \cstar-algebras. Unfortunately, this class does not have amalgamation: Let $A=\mathbb C\oplus \mathbb C$ and $B=M_3$, and consider the maps $\varphi_1\colon A\to B$ and $\varphi_2\colon A\to B$ given by
\[
\varphi_1((a,b))=\begin{pmatrix}
a & 0 & 0\\
0 & a & 0\\
0&0&b
\end{pmatrix}\text{ and }\varphi_2((a,b))=\begin{pmatrix}
a & 0 & 0\\
0 & b & 0\\
0&0&b
\end{pmatrix}.
\]
A trace argument shows that one cannot amalgamate these two maps: if $\psi_1$ and $\psi_2$ are two unital $^*$-homomorphisms whose range is a matrix algebra $M_k$, then, with $\tau$ the canonical trace on $M_k$, one verifies that 
\[
\frac{2}{3}=\tau(\psi_1\circ\varphi_1((1,0)))\neq \tau(\psi_2\circ\varphi_2((1,0)))=\frac{1}{3}.
\]
The obstructions to amalgamation above are of tracial nature. By adjoining the trace to the language, and only considering trace preserving $^*$-homomorphisms (see \S\ref{ss.blocks}), in \cite[3.8 and 3.9]{EFH:Fraisse} it was shown that certain subclasses of finite-dimensional \cstar-algebras are indeed \Fraisse classes. All of the \Fraisse limits arising this way are simple AF algebras.

On the other side of the spectrum are AF algebras which are surjectively universal, and therefore must contain an abundance of ideals. These were treated, by using a different approach not relying on tracial information, by Ghasemi and Kubi\`s in \cite{GhasemiKubis}. A definition:
\begin{definition}
Let $A$ and $B$ be \cstar-algebras. A map $\varphi\colon A\to B$ is \emph{left-invertible} if there is a $^*$-homomorphism $\pi\colon B\to A$ such that $\pi\circ\varphi$ is the identity on $A$. 
\end{definition}

The following is the main content of \cite{GhasemiKubis}:
\begin{theorem}
Let $\mathcal C$ be a class of finite-dimensional which is closed under taking direct sums and ideals, whose maps are left-invertible embeddings (not necessarily unital). Then $\mathcal C$ is a \Fraisse class.
\end{theorem}
By considering $\mathcal C$ be the class of all finite-dimensional \cstar-algebras, the \Fraisse limit of such class $A_{\mathcal C}$ is a separable AF algebra which surjects onto any separable AF algebra. Further, its trace space is affinely homeomorphic to the Poulsen simplex $\mathbf P$ (see \S\ref{ss.Poulsen}).

\subsubsection{Some classes of small algebras}\label{ss.blocks}
Let us introduce certain classes of \cstar-algebras we will use as building blocks. 

Some notations: $0_k$ denotes the $0$ matrix in $\mathrm{M}_k$, likewise $1_k$ is the identity. If $a\in M_k$ and $b\in M_{k'}$ are matrices, let $\diag(a,b):=\begin{pmatrix}
a &0 \\
0 &b
\end{pmatrix}\in M_{k+k'}$. This naturally generalises to $\diag(a_1,\ldots,a_n)$.

\begin{definition}
 If $n\in\mathbb N$, let 
\[
C_n:=C([0,1],M_n).
\]
We will refer to these as 
 \emph{interval algebras}.

If $p$ and $q$ are natural numbers, by identifying $M_{pq}$ with $M_p\otimes M_q$, we define
\[
\mathcal Z_{p,q}:=\{f\in C_{pq}\mid f(0)\in 1_p\otimes M_q\text{ and } f(1)\in M_p\otimes 1_q\}.
\]
We refer to these as 
 \emph{dimension drop algebras}.
\end{definition}
 Notice that $\mathcal Z_{p,q}$ is projectionless if and only if $p$ and $q$ are coprime (see \cite[Lemma 2.2]{JiangSu}). Coprime dimension drop algebras have $K$-theory $(\mathbb Z,0)$, where $1$ is the class of the identity\footnote{We do not define $K$-theory in this article. For the basics of $K$-theory, please see \cite{Rordam.Ktheory}.}.

Another interesting class of subalgebras of interval algebras was identified in \cite{Razak}. 
\begin{definition}
If $n,k\in\NN$, let
\begin{eqnarray*}
A_{n,k}:=\{f\in C([0,1],M_{nk})\mid \exists a\in M_k(f(0)=&\diag(\underbrace{a,\ldots,a}_n)\text{ and }\\ f(1)=&\diag(\underbrace{a,\ldots,a}_{n-1},0_k))\}.
\end{eqnarray*}
The algebras $A_{n,k}$ are called 
 \emph{Razak blocks}.

Let
\begin{eqnarray*}
B_{n,k}:=\{f\in C([0,1],M_{2nk})\mid \exists a,b\in M_k (f(0)=&\diag(\underbrace{a,\ldots,a}_n, \underbrace{b,\ldots,b}_n)\text{ and }\\ f(1)=&\diag(\underbrace{a,\ldots,a}_{n-1},0_k, \underbrace{b,\ldots,b}_{n-1},0_k))\}.
\end{eqnarray*}
The algebras $B_{n,k}$ are called 
 \emph{generalised Razak blocks}.
\end{definition}
The following summarises a few of the basic properties of these objects (see \cite[Proposition 2.8]{JacelonVignati}). Recall that a \cstar-algebra $A$ is 
 projectionless if it has no proper projections. $A$ is stably projectionless if $A\otimes\mathcal K$ is projectionless.

 \begin{proposition}\label{prop:Ktheory}
Let $n,k\in\NN$. Then
\begin{enumerate}
\item\label{noprojc1} Razak blocks and generalised Razak blocks are stably projectionless, but every proper quotient of each of them has a nonzero projection;
\item\label{noprojc2} nonzero $^*$-homomorphisms whose domain and codomain are a Razak block or a generalised Razak block are injective;
\item\label{noprojc3} $K_*(A_{n,k})=0$, $K_0(B_{n,k})\cong\zet$ and $K_1(B_{n,k})=0$.
\end{enumerate}
\end{proposition}

By Proposition~\ref{prop:Ktheory}, any limit of Razak blocks has trivial $K$-theory. On the other hand, limits of generalised Razak blocks can have different $K_0$-groups. Since $K_0(B_{n,k})\cong\mathbb Z$, a $^*$-homomorphism between generalised Razak blocks $\varphi$ induces a group endomorphism $\varphi_*:\mathbb Z\to \mathbb Z$. 
\begin{definition}
Let $\varphi$ be a $^*$-homomorphism between generalised Razak blocks, and let $k\in\mathbb Z$. We say that $\varphi$ has \emph{$K$-theory $k$} if $\varphi_*(1)=k$.
\end{definition}

Before giving the definition of the classes of interest, let us look at traces on interval algebras, dimension drop algebras, and (generalised) Razak blocks. 
\begin{definition}
If $A$ is a \cstar-algebra, a state $\tau$ on $A$ is a 
 \emph{trace} if $\tau(ab)=\tau(ba)$ for all $a,b\in A$. We denote by $T(A)$ the trace space of $A$. If $A$ and $B$ are \cstar-algebras, $\sigma\in T(A)$ and $\tau\in T(B)$, we say that a $^*$-homomorphism $\varphi\colon A\to B$ sends $\sigma$ to $\tau$, and write $\varphi\colon (A,\sigma)\to (B,\tau)$, if $\sigma(a)=\tau(\varphi(a))$ for all $a\in A$.
\end{definition}

As $M_n$ has a unique trace $\tau_n$, traces on $C_n$ correspond to Radon probability measures on $[0,1]$: if $\mu$ is a measure on $[0,1]$, let 
\[
\tau_\mu(f)=\int\tau_n(f(t))d\mu(t).
\]
$\tau_\mu$ is a trace on $C_n$. Conversely, given a trace $\tau$ on $C_n$, we can construct a measure $\mu_\tau$ by 
\[
\mu_\tau(O)=\sup_{f:\supp(f)\subseteq O, \norm{f}\leq 1}\tau(ff^*),
\]
for any open set $O\subseteq[0,1]$. It is routine to check that $\tau=\mu_{\tau_\mu}$ and $\mu=\mu_{\tau_\mu}$. With this association, we can borrow measure theoretic terminology and refer to faithful diffuse traces for those whose associated measure is faithful and diffuse. (Recall that a measure is diffuse if every measurable set of nonzero measure can be split in two measurable subsets each of nonzero measure. A measure is faithful if no nonempty open set has zero measure.) We denote by $T_{fd}(C_n)$ the set of faithful diffuse traces on $C_n$. 

The same reasoning applies to dimension drop algebras and (generalised) Razak's blocks; we will write $T_{fd}(\mathcal Z_{p,q}), T_{fd}(A_{n,k})$ and $T_{fd}(B_{n,k})$ for the corresponding sets of 
 faithful diffuse traces.

\begin{definition}
Let $A$ and $B$ be either interval algebras, dimension drop algebras, Razak blocks, or generalised Razak blocks. Let $\varphi\colon A\to B$ be a $^*$-homomorphism. We say that $\varphi$ \emph{pulls back faithful diffuse traces} if there are $\sigma\in T_{fd}(A)$ and $\tau\in T_{fd}(B)$ such that $\varphi\colon (A,\sigma)\to (B,\tau)$.
\end{definition}

Notice that $^*$-homomorphisms pulling back faithful diffuse traces are automatically injective. Furthermore, the choice of $\tau$ in the definition above is not relevant: in the above setting, if $\varphi$ pulls back a faithful diffuse trace to a faithful diffuse trace then the pullback of any faithful diffuse trace is such. Before proceeding, we record the following, which shows that all diffuse faithful traces are the same up to isomorphism.

\begin{lemma}\label{lemma:movingtraces}
Let $n\in\NN$ and $\sigma,\tau\in T_{fd}(C_n)$. Then there is an isomorphism $(C_n,\sigma)\to(C_n,\tau)$. The same applies to dimension drop algebras, and (generalised) Razak blocks.
\end{lemma}
\begin{proof}
For $t\in [0,1]$, define $\varphi(f)(t)=f(s)$, where $s\in [0,1]$ is the unique point such that $\mu_\sigma([0,s])=\mu_\tau([0,t])$. This is the required isomorphism.
\end{proof}

\begin{definition}
The language of tracial \cstar-algebras $\mathcal L_{\mathrm{C}^*,\tau}$ is the language of \cstar-algebras expanded by a unary predicate $\tau$ whose modulus of uniform continuity is the identity. 
\end{definition}
We are ready to define our classes. 
\begin{definition}\label{def:classes}
All classes are considered in the langage of tracial \cstar-algebras, that is, structures are pairs $(A,\tau)$. Let $\bar p$ be a supernatural number. Then
\begin{itemize}
\item $\mathcal K_{[0,1],\bar p}$ is the class whose objects have the form $(C_n,\tau)$, where $n$ divides $\bar p$ and $\tau\in T_{fd}(C_n)$. Embeddings are $^*$-homomorphisms which pull back faithful diffuse traces;
\item $\mathcal K_{\mathcal Z}$ is the class whose objects are pairs $(\mathcal Z_{p,q},\tau)$ where $p$ and $q$ are coprime and $\tau\in T_{fd}(\mathcal Z_{p,q})$. Embeddings are $^*$-homomorphisms which pull back faithful diffuse traces;
\item $\mathcal K_{\mathcal W}$ is the class whose objects are pairs $(A_n,\tau)$ where $A_n$ is a Razak block and $\tau\in T_{fd}(A_n)$. Embeddings are $^*$-homomorphisms which pull back faithful diffuse traces;
\item $\mathcal K_{\mathcal Z_0}$ is the class whose objects are pairs $(B_n,\tau)$ where $B_n$ is a generalised Razak block and $\tau\in T_{fd}(B_n)$. Embeddings are $^*$-homomorphisms which pull back faithful diffuse traces;
\item $\mathcal K_{r,1}$ is the subclass of $\mathcal K_{\mathcal Z_0}$ where only embeddings whose $K$-theory belongs to $\{-1,1\}$ are considered;
\item $\mathcal K_{r,\bar p}$ is the subclass of $\mathcal K_{\mathcal Z_0}$ where only embeddings whose $K$-theory divides $\bar p$ are considered;
\end{itemize}

\end{definition}
\begin{theorem}\label{thm:CstarFraisse}
In the language of tracial \cstar-algebras, each of the classes in Definition~\ref{def:classes} is a \Fraisse class. Moreover
\begin{itemize}
\item The \Fraisse limit of $\mathcal K_{[0,1],\bar p}$ is $M_{\bar p}$, the UHF algebra associated to the supernatural number $\bar p$;
\item The \Fraisse limit of $\mathcal K_\mathcal Z$ is the 
 Jiang-Su algebra $\mathcal Z$;
\item The \Fraisse limit of $\mathcal K_\mathcal W$ is the 
 Jacelon algebra $\mathcal W$, which is also the \Fraisse limit of the class $\mathcal K_{r,1}$;
\item The \Fraisse limit of $\mathcal K_{\mathcal Z_0}$ is the algebra 
 $\mathcal Z_0$;
\item The \Fraisse limit of $\mathcal K_{r,\bar p}$ is the algebra $\mathcal Z_0\otimes M_{\bar p}$, where $M_{\bar p}$ is the UHF algebra associated to the supernatural number $\bar p$.
\end{itemize}

\end{theorem}

Theorem~\ref{thm:CstarFraisse} is a result of the work of many hands. The class $\mathcal K_\mathcal Z$ was treated in \cite{EFH:Fraisse}. The proof in \cite{EFH:Fraisse} relies on the classification of morphisms from $1$-dimensional NCCW complexes into stable rank one \cstar-algebras (and in particular, into $\mathcal Z$) of Robert's (\cite{Robert:2010qy}), which served to proved the NAP. A by hand proof of the NAP for dimension drop algebras, not relying on any classification result, was written in \cite{Masumoto.FraisseZ} (see also \cite{Masumoto.Real}). The classes $\mathcal K_{[0,1],\bar p}$ were treated in \cite[\S3]{Masumoto.FraisseZ}. The other classes, e.g., those giving rise to nonunital \cstar-algebras, were treated in \cite{JacelonVignati}.

The proofs (except the one from \cite{EFH:Fraisse}) follow a similar pattern, which we will now sketch in the case of the classes $\mathcal K_{[0,1],\bar p}$. While the proofs for interval algebras, dimension drop algebras, and Razak blocks carry a similar amount of technical difficulties, this is not the case for generalised Razak blocks, where the proof is quite intricate. They main difficulty is always in proving the Near Amalgamation Property. 

\subsubsection{The proof}
This section is dedicated to the proof of Theorem~\ref{thm:CstarFraisse} for the classes $\mathcal K_{[0,1],\bar p}$. We will not prove the individual lemmas leading to the proof of the theorem; full proofs can be found in \cite{Masumoto.FraisseZ}.

For simplicity, we assume $\bar p$ is the largest supernatural number, i.e., $\bar p=\prod_{p\text{ prime}}p^\infty$. We write $\mathcal K$ for $\mathcal K_{[0,1],\bar p}$.

Recall that $C_n=C([0,1],M_n)$, so objects in $\mathcal K$ are pairs $(C_n,\sigma)$, with $\sigma\in T_{fd}(C_n)$, and morphisms are $^*$-homomorphisms which pull back faithful diffuse traces. As already mentioned, all such maps are injective. It is an easy exercise to show that this class has the WPP. Another easy exercise, using Lemma~\ref{lemma:movingtraces}, gives that $\mathcal K$ has the JEP. We are left to show that $\mathcal K$ has the NAP and to analyse its \Fraisse limit.

\begin{definition}
Let $n$ and $m$ be natural numbers, where $n$ divides $m$. Let $j=\frac{m}{n}$. A $^*$-homomorphism $\varphi\colon C_n\to C_m$ is \emph{diagonal} if there are continuous maps $\xi_1,\ldots,\xi_j\colon [0,1]\to[0,1]$ with $\xi_1\leq\cdots\leq\xi_j$ and a unitary $u\in C_m$ such that for every $f\in C_n$ and $t\in [0,1]$ we have 
\[
\varphi(f)(t)=u(t)\diag(f(\xi_1(t)),\ldots,f(\xi_j(t)))u^*(t).
\]
The maps $\xi_1,\ldots,\xi_j$ are said to be \emph{associated} to $\varphi$
\end{definition}
First, we show that we can reduce to the case of diagonal maps.
\begin{lemma}
Let $n,m\in\NN$ and suppose that $n$ divides $m$. Let $F\subseteq C_n$ be finite and $\e>0$. Suppose that $\varphi\colon C_n\to C_m$ is a unital $^*$-homomorphism. Then there is a diagonal $\psi\colon C_n\to C_m$ such that
\[
\norm{\varphi(a)-\psi(a)}<\e, \text{ for all } a\in F.
\]
Moreover, if $\sigma\in T(C_n)$ and $\tau\in T(C_m)$ are such that $\varphi$ pulls $\tau$ back to $\sigma$, then so does $\psi$.
\end{lemma}

A useful technical tool is the definition of diameter.
\begin{definition}
Let $\varphi\colon C_n\to C_m$ be a diagonal map with associated maps $\{\xi_i\}$, for $i\leq\frac{m}{n}$. The \emph{diameter} of $\varphi$ is the quantity
\[
\max_i\sup_{s,t\in [0,1]} |\xi_i(s)-\xi_i(t)|.
\]
\end{definition}

The proof of amalgamation passes through the following two propositions. The first one uses the fact that continuous maps $[0,1]\to[0,1]$ are uniformly continuous, while the second one uses that elements of $C_n$ are uniformly continuous (when viewed as map from $[0,1]$ to a bounded ball of $M_n$.)
\begin{proposition}
Let $n,m\in\NN$ and suppose that $n$ divides $m$. Let $\varphi\colon C_n\to C_m$ be a diagonal map. Fix $\e>0$. Then there is $\delta>0$ such that for all $k$ and all diagonal maps $\psi\colon C_m\to C_k$, if the diameter of $\psi$ is $<\delta$, then the diameter of $\psi\circ\varphi$ is $<\e$.
\end{proposition}

\begin{proposition}\label{prop:key}
Let $n\in\NN$, $F\subseteq C_n$ be finite, and $\e>0$. Then there is $\delta>0$ such that for all $m\in \NN$ with $n$ dividing $m$, $\sigma\in T_{fd}(C_n)$ and $\tau\in T_{fd}(C_m)$, if $\varphi_1,\varphi_2\colon (C_n,\sigma)\to (C_m,\tau)$ are diagonal maps of diameter $<\delta$, then there is a unitary in $C_m$ such that 
\[
\norm{u\varphi_1(a)u^*-\varphi_2(a)}<\e, \text{ for all } a\in F.
\]
\end{proposition}
Proposition~\ref{prop:key} is the key that opens the door to amalgamation. This is also where, in the case of other classes of algebras, technical issues arise. In fact, the main problem there is choosing (for example, in case $\varphi_1$ and $\varphi_2$ are maps between dimension drop algebras $A$ and $B$) the unitary $u$ in such a way that $u\varphi_1[A]u^*$ remains inside $B$. The amount of work required for this step in case of dimension drop algebras and Razak blocks is acceptable, but technicalities become a real issue in case of generalised Razak blocks.

\begin{theorem}
The class $\mathcal K$ has the Near Amalgamation Property.
\end{theorem}
\begin{proof}
Let $\lambda$ be the Lebesgue measure on $[0,1]$. We denote by $\tau_\lambda$ its associated trace on $C([0,1])$. This gives a trace (again denoted by $\tau_\lambda$) on each $C_n$. By Lemma~\ref{lemma:movingtraces} and the JEP, it is enough to show that when given $n,m\in\NN$ with $n$ dividing $m$, a finite $F\subseteq C_n$, $\e>0$ and two $^*$-homomorphisms $\varphi_1,\varphi_2\colon (C_n,\tau_\lambda)\to (C_m,\tau_\lambda)$, then one can find a large enough $k$ and two $^*$-homomorphisms $\psi_1,\psi_2\colon (C_m,\tau_\lambda)\to (C_k,\tau_\lambda)$ such that 
\[
\norm{\psi_1\circ\varphi_1(a)-\psi_2\circ\varphi_2(a)}<\e, \text{ for all } a\in F.
\]
First, one finds $\delta$ so that it satisfies Proposition~\ref{prop:key} for $F$ and $\e$. Secondly, one finds $\delta'$ so small that whenever we have a map of diameter $<\delta'$, then the diameter of both $\psi\circ\varphi_1$ and $\psi\circ\varphi_2$ is $<\delta$. Finding maps of arbitrarily small diameter which pull back faithful diffuse traces is easy for interval algebras, but requires some work for dimension drop and (generalised) Razak blocks; in this case, the thesis follows by applying again Lemma~\ref{lemma:movingtraces}.
\end{proof}
Let $M$ be the \Fraisse limit of $\mathcal K$. To show that $M$ is isomorphic to the universal UHF algebra $M_{\bar p}$, first one shows that $M$ is simple, monotracial, and AF (that is, a limit of finite-dimensional \cstar-algebras). All such verifications are not difficult, assuming the reader is comfortable with \cstar-algebra theory (one can follow the ideas of \cite[\S4]{Masumoto.FraisseZ}). Then, one computes $K_0(M)$, the $K_0$-group of $M$, and shows that such group must be divisible, and therefore equal to $\mathbb Q$. The result then follows from Elliott's classification theorem (\cite{Elliott.intertwining}), asserting that AF algebras are classified by their $K_0$ group. 

For the other classes (and limits) treated in Theorem~\ref{thm:CstarFraisse}, one can either appeal to classification results for amenable \cstar-algebras, or show that the sequences explicitly defining the objects $\mathcal Z$, $\mathcal W$, and $\mathcal Z_0$ (as given in \cite{JiangSu}, \cite{Jacelon:W} and \cite{JacelonVignati}) are generic (i.e., their limit is the \Fraisse limit).

\subsubsection{Strong self-absorption of $\mathcal Z$ by \Fraisse theoretic methods}
One of the most important features of $\mathcal Z$ is that $\mathcal Z$ is isomorphic to $\mathcal Z\otimes\mathcal Z$ in a very strong sense (this is called 
 \emph{strong self-absorption}, see \cite{JiangSu}). 

Can one prove strong self-absorption of $\mathcal Z$ with \Fraisse theoretic methods? If done in a clumsy way (e.g., by adding to the class $\mathcal K_\mathcal Z$ all possible tensor products of dimension drop algebras and considering all maps which pull back faithful diffuse trace), amalgamation fails, as one is capable to build $K_1$-related obstructions in $[0,1]^2$ which do not appear in $[0,1]$. Ghasemi, in \cite{Ghasemi.Fraisse} was more careful, and restricting the class of maps considered, showed that one can construct a \Fraisse class which is rich enough to have two generic sequences, one having $\mathcal Z$ as its limit, and the other one that has $\mathcal Z\otimes\mathcal Z$.

While in general, even without the use of \Fraisse theory, it is not difficult to show that $\mathcal Z$ tensorially absorbs itself, the same cannot be said about Jacelon's $\mathcal W$. It is known that $\mathcal W$ and $\mathcal W\otimes\mathcal W$ are isomorphic (by deep, long, and convoluted classification methods e.g. \cite{GongLin.Class2}), but a direct proof has yet to be found. There have been unsuccessful, yet serious, attempts to use \Fraisse theory to show that $\mathcal W\cong\mathcal W\otimes\mathcal W$ without the aid of classification tools. This problem seems difficult and deserves to be stated. Similar problems can be stated for the algebra $\mathcal Z_0$.
\begin{problem}
Use \Fraisse theory to show that $\mathcal W\cong\mathcal W\otimes\mathcal W$, that is, find a \Fraisse class which contains enough Razak blocks and their tensor products, which has two generic sequences naturally giving $\mathcal W$ and $\mathcal W\otimes\mathcal W$.
\end{problem}
\section{\Fraisse limits in functional analysis}\label{s.fa}
In this section we record \Fraisse theoretic results in functional analysis. We will show below that the Urysohn space, the Hilbert space $\ell_2(\NN)$, the Gurarij space, the Poulsen simplex, and the noncommutative versions of the latter two are \Fraisse limits of suitable classes of objects. The Urysohn space, $\ell_2(\NN)$, and the Gurarij space were treated in \cite{BY:Fraisse} (see also \cite{KubisSolecki}), while the Poulsen simplex, its noncommutative version and the noncommutative Gurarij space were treated in \cite{Lupini.fa} and \cite{Lupini.NCGurarij}. The Poulsen simplex was also studied in \cite{TornConley.Fraisse}.

\subsection{The Urysohn space}

A remarkable date for topologists is August 3rd, 1924, when, in a letter to Hausdorff, Urysohn announced the construction of a separable complete metric space containing isometrically any other separable metric space and satisfying a strong homogeneity property.\footnote{The letter does not contain any detail of the construction, which Hausdorff re-did himself.} 
\begin{definition}
A separable metric space $X$ is \emph{ultrahomogeneous} if every isometry between finite subsets of $X$ can be extended to a self-isometry of $X$.
\end{definition}
\begin{definition}
A 
 \emph{Urysohn space} is an ultrahomogeneous separable complete metric space  in which every separable metric space isometrically embeds.
\end{definition}
Urysohn shows that a Urysohn space exists, and further, that it is unique up to isometry. It does make sense therefore to talk about \emph{the} Urysohn space $\mathbf U$. The Urysohn space was somewhat rediscovered by Katetov in \cite{Katetov}, who constructed a space which is universal for metric spaces of arbitrary density character. 

With ultrahomogeneity for the class of finite metric spaces in our pocket, it remains to formalise the proper class of objects we are interested in, and their maps, to recognise $\mathbf U$ as a \Fraisse limit. This was done in \cite{BY:Fraisse}.
\begin{theorem}
Let $\mathcal K_{\mathbf U}$ be the class of finite metric spaces with maps being isometric embeddings. Then $\mathcal K_{\mathbf U}$ is a \Fraisse class whose \Fraisse limit is $\mathbf U$.
\end{theorem}
Similarly to the Urysohn space, the Urysohn sphere, which plays the same role as $\mathbf U$ does for metric spaces whose diameter is bounded by $1$, can be obtained as a \Fraisse limit.
\subsection{The Gurarij spaces}
The language of metric spaces is the simplest possible language of metric structures, as only the metric is in the language. In the discrete setting, this correspond to the fact that a countable infinite set is the \Fraisse limit of finite sets. Further results are obtained by adding a bit more structure to the language.
\begin{definition}
A 
 \emph{Gurarij space} is a separable Banach space $\mathbf G$ having the property that, for any $\e>0$, any finite-dimensional Banach spaces $E\subseteq F$, any isometric linear embedding $E\to\mathbf G$ can be extended to a linear embedding $\varphi\colon F\to\mathbf G$ satisfying $\norm{\varphi}\cdot\norm{\varphi^{-1}}\leq 1+\e$.
\end{definition}
The existence of a Gurarij space was proved in \cite{Gurarij}, but its uniqueness had to wait for \cite{Lusky}, where it was shown by the use of deep techniques of Lazar and Lindenstrauss (\cite{LazarLinde}). Let $\mathbf G$ be the unique (up to linear isometries) Gurarij space. After an unpublished model theoretic proof of Henson, Kubi\`s and Solecki (\cite{KubisSolecki}) proved uniqueness of $\mathbf G$ by essentially showing that it is a \Fraisse limit. Their argument was formalised in \cite{BY:Fraisse} when the general theory was developed.
\begin{theorem}
Let $\mathcal K_{\mathbf G}$ be the class of finite-dimensional Banach spaces with maps being isometric linear embeddings. Then $\mathcal K_{\mathbf G}$ is a \Fraisse class whose \Fraisse limit is $\mathbf G$.
\end{theorem}
We further add structure, and consider Hilbert spaces. The class of finite-dimensional Hilbert spaces together with Hilbert space isometries has the amalgamation property, hence it is a \Fraisse class. The \Fraisse limit is the unique separable infinite-dimensional Hilbert space, $\ell_2(\NN)$.

We now consider a noncommutative analog of Gurarij's $\mathbf G$. (For more on operator spaces, see Sinclair's article in this volume.)
\begin{definition}
An 
 \emph{operator space} is a Banach subspace of $\mathcal B(H)$ for some Hilbert space $H$. Pertinent for linear maps between operator spaces is their completely bounded norm. If $E$ and $F$ are operator spaces and $\varphi\colon E\to F$ is a linear map, we denote by $\varphi_n\colon M_n(E)\to M_n(F)$ its amplification, where $M_n(E)$ is the operator space of $n\times n$ $E$-valued matrices. The \emph{completely bounded norm} of $\varphi$ is given by
\[
\norm{\varphi}_{cb}:=\sup_n\norm{\varphi_n}.
\]
\end{definition}
Notions of amenability often pass through internal approximations via finite (or finite-dimensional) objects in the category. The following is the appropriate notion for operator spaces; it is strongly linked to the notion of exactness of \cstar-algebras (\cite{Pisier}).
\begin{definition}
Let $c\in\mathbb R$. An operator space $X$ is $c$-\emph{exact} if for all $\e>0$ and any finite-dimensional $E\subseteq X$, there is $n$ and an operator subspace $F\subseteq M_n$ such that there is an isomorphism $\varphi\colon E\to F$ with $\norm{\varphi}_{cb}\cdot\norm{\varphi^{-1}}_{cb}<c+\e$. An operator space is 
 \emph{exact} if it is $c$-exact for some $c\in\mathbb R$.
\end{definition}

\begin{definition}
A separable operator space $\mathbf {NG}$ is a 
 \emph{noncommutative Gurarij space} if it satisfies the following: suppose a finite-dimensional operator space $F$ is $1$-exact, $E$ is a subspace of $F$, and $E'$ is a subspace of $\mathbf {NG}$ which is isomorphic to $E$ via a completely bounded $\varphi$ whose inverse is completely bounded. Then, for every $\e>0$, there is a subspace $F'\subseteq \mathbf {NG}$ which contains $E'$, and an isomorphism $\tilde\varphi\colon F \to F'$ which extends $\varphi$ and satisfies $\norm{\tilde\varphi}_{cb}\cdot\norm{\tilde\varphi^{-1}}_{cb}\leq(1+\e)\norm{\varphi}_{cb}\cdot\norm{\varphi^{-1}}_{cb}$.
\end{definition}

A noncommutative Gurarij space exists by \cite{Oikhberg.Gurarij}, where a weaker version of universality, for those operator spaces which can be locally approximated by finite-dimensional \cstar-algebras, was proved. Again in \cite{Oikhberg.Gurarij}, a weak form of uniqueness was established: it was shown that any two noncommutative Gurarij spaces are $c$-isomorphic, for every $c>1$, that is, one can find an isomorphism $\varphi$ satisfying $\norm{\varphi}_{cb}\cdot\norm{\varphi^{-1}}_{cb}<c$. Later, uniqueness (up to linear complete isometries) was shown in \cite{Lupini.NCGurarij} with \Fraisse theoretic methods, while also obtaining the appropriate homogeneity in this setting.
\begin{theorem}
Let $\mathcal K_{OS}$ be the class of finite-dimensional $1$-exact operator spaces with maps being completely isometric embeddings. Then $\mathcal K_{OS}$ is a \Fraisse class whose \Fraisse limit is the noncommutative Gurarij space $\mathbf{NG}$, which is therefore unique (up to linear complete isometries) and universal for all separable $1$-exact operator spaces.
\end{theorem} 

\subsection{The Poulsen simplex and the Poulsen operator system}\label{ss.Poulsen}

Before getting into the study of the Poulsen simplex, we review a few definitions from Choquet's theory. We direct the reader to \cite{Phelps} for an introduction to this topic.

\begin{definition}
Let $X$ be a compact subset of a locally convex space $E$. Let $\mu$ be a Borel probability measure on $X$ and let $x\in X$. We say that $x$ is \emph{the barycenter} of $\mu$ if for every continuous linear functional $f$ on $E$ we have that $f(x)=\int_Xfd\mu$.
\end{definition}
The following question was instrumental for the origins and developments of Choquet's theory: If $X$ is a compact convex subset of a locally convex space $E$ and $x\in X$, does there exist a probability measure $\mu$ on $X$ which is supported on the extreme points of $X$ and which has $x$ as its barycenter? If it exists, is it unique?

In case $X$ is metrizable, the first question has a positive answer (\cite{Choquet1}). The following is one of the equivalent definitions of Choquet simplex one finds in literature (see \cite{Choquet2}).
\begin{definition}
Let $X$ be a metrizable compact convex subset of a locally convex space $E$. $X$ is said to be a 
 \emph{Choquet simplex} if for every $x\in X$ there is a unique probability measure $\mu$ on $X$ which is supported on the extreme points of $X$ and has $x$ as its barycenter\footnote{In case of nonmetrizable compact convex spaces, the extreme boundary does not have to be Borel, so we ask for uniqueness among measures vanishing on any Baire subset of $X\setminus\delta X$.}.
\end{definition}

Compact convex sets are dual to function systems (see \S2 of Sinclair's article):

\begin{definition}
A 
 \emph{function system} is a pointed ordered vector space $(V,u)$ where $u$ is a positive Archimedean order unit.
\end{definition}

If $(V,u)$ is a function system, $V$ can be endowed with a norm defined by 
\[
\norm{x}=\inf_{r>0} \{-ru\leq x\leq ru\}.
\] 
Isomorphisms in this class are surjective linear isometries preserving the order unit. From the definition of the norm, one can define the space of states on a function system $(V,u)$, by considering all of those positive linear functionals on $V$ which have norm $1$. This is a compact convex set. Conversely, if $K$ is a compact convex set, let $A(K)$ be the space of real affine continuous functions on $K$. $A(K)$ is a function system, with the order given by function domination and the order unit being the identity function. Kadison's representation theorem, stated below, gives the required duality (see \cite[Theorem II.1.8]{Alfsen.Book}):\begin{theorem}
The assignment which associated a compact convex set $K$ to the function system $A(K)$ is a controvariant equivalence of categories.
\end{theorem}

Under this equivalence, metrizable Choquet simplices correspond to those separable function systems which are in addition Lindenstrauss spaces (preduals of $L_1$ spaces).

\begin{definition}
A 
 \emph{Poulsen simplex} is a metrizable Choquet simplex whose extreme points are dense.
\end{definition}

A Poulsen simplex $\mathbf P$ was constructed in \cite{Poulsen}. In \cite{Lind-Olsen-Sternfeld}, it was shown that $\mathbf P$ is the unique Poulsen simplex (up to affine homeomorphism). Methods used in approaching the study of $\mathbf P$ are closely related to those used for the Gurarij space $\mathbf G$ (see e.g. \cite{Olsen}).

In \cite{TornConley.Fraisse}, and later \cite{Lupini.fa}, the Poulsen simplex $\mathbf P$ was recognised as the space of states of a \Fraisse limit of the class of function systems.

In particular, in \cite{TornConley.Fraisse}, it was showed that $A(\mathbf P)$ is the unique separable function system that is approximately homogeneous and universal for separable function systems. In a sense, $A(\mathbf P)$ has the same properties as Gurarij's $\mathbf G$ does, but in the category of function systems.

\begin{theorem}[{\cite{TornConley.Fraisse}}, see also {\cite[\S6 and 7]{Lupini.fa}}]
The class of finite-dimensional function systems where maps are order unit preserving, linear, isometric embeddings is a \Fraisse class whose \Fraisse limit is $A(\mathbf P)$.
\end{theorem}

The universality properties (for function spaces) one obtains for $A(\mathbf P)$ translate in the language of Choquet simplices in that a metrizable compact convex set is a Choquet simplex if and only if it is affinely homeomorphic to a closed face of $\mathbf P$.

We study now the noncommutative version of Poulsen's $\mathbf P$.

\begin{definition}
An 
 \emph{operator system} is an operator subspace of $\mathcal B(H)$ which is closed by adjoints and contains the unit of $\mathcal B(H)$.
\end{definition}

As operator spaces can be viewed as noncommutative Banach spaces, the work of Arveson (\cite{Arveson1} and \cite{Arveson2}) shows that operator systems provide the natural noncommutative analogues of function systems. In fact, function systems are those operator systems contained in a unital abelian \cstar-algebra. For this reason, the space introduced below is considered to be the noncommutative version of Paulsen's $\mathbf P$.

In the same way to a function system one can associate its state space, operator systems give rise to noncommutative Choquet simplices. Noncommutative Choquet simplices are special cases of noncommutative compact convex spaces. The latter were introduced in \cite{KennedyDavidson} and extensively studied in \cite{KennedyShamovich}. We will not give the precise definition of a noncommutative Choquet simplex here, as it is fairly technical, and we direct the reader to \cite[\S4]{KennedyShamovich} for the details. The important correspondence to retain is the existence of an equivalence of categories between the category of noncommutative compact convex spaces and that of operator systems, associating to each noncommutative compact convex space $X$ an operator system $A(X)$. Again noncommutative Choquet simplices correspond to those operator systems satisfying Lindestrauss-like conditions.

In this setting, Kennedy and Shamovich proved in \cite{KennedyShamovich} the existence of a metrizable 
 noncommutative Poulsen simplex $\mathbf {NP}$ having \emph{dense extreme points} (again, we decided not to be precise here to avoid too many technicalities). Even though it is not yet known whether this denseness-like property alone gives uniqueness of $\mathbf{NP}$ among metrizable noncommutative Choquet simplices, the corresponding (via the equivalence of categories of \cite{KennedyDavidson}) operator system associated to $\mathbf{NP}$ has been recognised as a \Fraisse limit.

\begin{theorem}[{\cite[\S7]{Lupini.fa}}]
The class of finite-dimensional exact operator systems where maps are unital completely isometric linear embeddings\footnote{Notice that these automatically preserve the adjoint.} is a \Fraisse class, whose \Fraisse limit is  $A(\mathbf {NP})$, the unique operator system having the noncommutative Poulsen simplex $\mathbf{NP}$ as its dual.
\end{theorem}

The space $A(\mathbf{NP})$ is the unique separable nuclear operator systems that is universal in the sense of Kirchberg and Wassermann (\cite{KirchWass}), that is, it contains a completely isometric copy of any separable exact operator system. Also, in the same fashion as for the Poulsen simplex, universality is expressed in that compact metrizable noncommutative Choquet simplices are exactly those homeomorphic to \emph{faces} of $\mathbf{NP}$.

\section{Connections with Ramsey theory}\label{s.ramsey}
In this section we review the continuous version of the KPT correspondence (\cite{KPT}), and we see how this was applied to study the automorphisms of some of the \Fraisse limits described in \S\ref{s.fa}, summarising some of the results of \cite{BartosovaRamsey1} and \cite{BartosovaRamsey2}. 

The main idea is the following: if $\mathcal K$ is a \Fraisse class of finitely-generated structures with \Fraisse limit $M$, then Ramsey theoretic properties of $\mathcal K$ correspond to dynamical informations about the group $\Aut(M)$. As any Polish group can be seen as the automorphism group of a \Fraisse limit (in an appropriately defined language, see \cite[Theorem 2.4.5]{Gao} and \cite[Theorem 6]{Melleray.Hjorth}, but also \cite[Chapter 5]{ThesisKaichouh}), \Fraisse theory has the potential of accessing certain dynamical features of Polish groups.

\begin{definition}
A Polish group $G$ is 
 \emph{extremely amenable} if every continuous action of $G$ on a compact metric space $X$ has a fixed point.
\end{definition}

As in \S\ref{s.prel}, a language for metric structure $\mathcal L$ and a class of finitely-generated $\mathcal L$-structures $\mathcal K$ are fixed. Recall that if $A$ belongs to $\mathcal K$, and $\bar a$ generates $A$, we use the notation $(A,\bar a)$ to record the tuple of generators.


Fix $(A,\bar a)$ and $B$ in $\mathcal K$. Let ${}^AB$ be the set of morphisms of $A$ into $B$. If $\varphi,\psi\in{}^AB$, set
\[
d_{\bar a}(\varphi,\psi):=d_B(\varphi(\bar a),\psi(\bar a))
\]
where $d_B$ is the metric on $B$. 
A 
 \emph{coloring} of ${}^AB$ is a $1$-Lipschitz map ${}^AB\to[0,1]$, where ${}^AB$ is considered with the metric $d_{\bar a}$.
\begin{definition}
A \Fraisse class $\mathcal K$ has the 
 \emph{approximate Ramsey property} if the following happens: for all $A$ and $B$ in $\mathcal K$, all finite $F\subseteq {}^AB$, and all $\e>0$, there is $C\in \mathcal K$ with the property that for every coloring $\gamma$ of ${}^AC$ there is $\varphi\in {}^BC$ such that 
\[
|\gamma(\varphi\circ\psi_1)-\gamma(\varphi\circ\psi_2)|<\e, \text{ for all }\psi_1,\psi_2\in F.
\]
\end{definition}
(First, notice that the definition above depends on the choice of generators of $A$. Second, in order the above definition is not vacuous, we assume that the empty function is $1$-Lipschitz, so that ${}^BC$ is nonempty.)

The following was proved in \cite{TsankovMelleray} for hereditary classes. A proof for non hereditary classes can be extracted from the proof contained in \cite{TsankovMelleray}, even though it has never been written down formally.
\begin{theorem}\label{thm:EA}
Let $\mathcal K$ be a \Fraisse class with \Fraisse limit $M$. The following are equivalent:
\begin{itemize}
\item $\mathcal K$ has the approximate Ramsey property;
\item the group $\Aut(M)$ is extremely amenable.
\end{itemize}
\end{theorem}

To a Polish group $G$, one associates its universal minimal flow. 
\begin{definition}
Let $G$ be a Polish group. A $G$-\emph{flow} is a compact Hausdorff space $X$ equipped with a continuous action of $G$ on $X$. A $G$-flow is \emph{minimal} if it has no proper subflows. The 
 \emph{universal minimal flow} $M(G)$ of $G$ is a minimal $G$-flow which maps continuously and equivariantly onto every minimal $G$-flow.
\end{definition}
The universal minimal flow is of great interest in topological dynamics; in fact, extreme amenability of $G$ is equivalent to the universal minimal flow $M(G)$ being a singleton. In the case of automorphism groups of \Fraisse limits, another `smallness' property, namely metrizability of $M(G)$, can be associated to Ramsey-like conditions on the class of interest. This was studied extensively by Zucker (see e.g. \cite{Zucker1}).

We quickly overview to which of the structures mentioned in \S\ref{s.opalg} and \S\ref{s.fa} Theorem~\ref{thm:EA} applies. 

Let us first focus on the \Fraisse classes mentioned in \S\ref{s.opalg}. In the abelian setting, while the dual Ramsey theorem of Graham and Rothschild \footnote{The dual Ramsey Theorem is a powerful pigeonhole principle which is equivalent to a factorization result for colorings of Boolean matrices, see \cite{GrahamRoth}.} can be used to show that the automorphism group of $C(2^\NN)$ (which is, the homeomorphism group of the Cantor set) is metrizable (and indeed homeomorphic to the Cantor set itself, see \cite{GlasnerWeiss}), the same is not known for $C(\mathbb P)$, $\mathbb P$ being Bing's pseudoarc (even though $\Aut(C(\mathbb P))$ is not extremely amenable, as the action of $\Aut(C(\mathbb P))$ on $\mathbb P$ itself has no fixed points, by homogeneity of $\mathbb P$.) 

In the nonabelian setting, not much in known. Using Ramsey type results for the class of matrix algebras (either with the operator norm or the trace norm), it is shown in \cite{EFH:Fraisse} that the automorphism groups of UHF algebras and of the hyperfinite II$_1$ factor $\mathcal R$ are extremely amenable. (Indeed, even more, their automorphism groups are L\'evy, see \cite[\S5 and \S6]{EFH:Fraisse}.) The situation of the automorphism group of $\mathcal Z$, $\mathcal W$, and $\mathcal Z_0$ is yet unclear; the reason for this is that rephrasing Ramsey like conditions on diagonal maps between dimension drop algebras or (generalised) Razak blocks turns out to be extremely technical.
\begin{problem}
Are the automorphism groups of $\mathcal Z$, $\mathcal W$, or $\mathcal Z_0$ extremely amenable? Do they have metrizable universal minimal flow?
\end{problem}

Much more is known for the structures described in \S\ref{s.fa}, mainly thanks to the work of \cite{BartosovaRamsey1} and \cite{BartosovaRamsey2}. In \cite{BartosovaRamsey1}, the authors study the Ramsey property for Banach spaces and Choquet simplices. Establishing the approximate Ramsey property for the class of finite-dimensional Banach spaces and function systems give the following:
\begin{theorem}
The following hold:
\begin{itemize}
\item The automorphism group of the Gurarij space $\mathbf G$ is extremely amenable.
\item If $F$ is a face of the Poulsen simplex $\mathbf P$, the group of automorphisms of $\mathbf P$ stabilising $F$ is extremely amenable.
\end{itemize} 
\end{theorem}
The main tool for establishing the above mentioned results is again the dual Ramsey Theorem. The same arguments, adapted to the proper setting, were replicated in \cite{BartosovaRamsey2} to show corresponding results for the noncommutative Gurarij space $\mathbf{NG}$ and the noncommutative Poulsen simplex $\mathbf{NP}$.

\bibliographystyle{plain} 
\bibliography{vignati_ref}
\printindex
\end{document}